\documentclass{amsart}

\usepackage{amsmath}
\usepackage{amssymb}
\usepackage{amscd}
\usepackage{amsthm}

\theoremstyle{definition}

\newtheorem{thm}{Theorem}[section]
\newtheorem{lem}[thm]{Lemma}
\newtheorem{definition}[thm]{Definition}
\newtheorem{rem}[thm]{Remark}

\newtheorem{prop}[thm]{Proposition}
\newtheorem{ex}[thm]{Example}

\newtheorem{cor}[thm]{Corollary}

\newcommand{\R}{\mathbb{R}}

\newcommand{\A}{\mathcal{A}}
\newcommand{\D}{\mathcal{D}}

\newcommand{\abs}[1]{\lvert {#1} \rvert}
\newcommand{\ang}[2]{\langle {#1},\, {#2}\rangle}

\newcommand{\Conf}{\mathrm{Conf}\,}

\newcommand{\fK}{\widetilde{\mathcal{K}}_n}

\newcommand{\id}{\mathrm{id}}

\newcommand{\Int}{\mathrm{Int}\,}

\newcommand{\K}{\mathcal{K}_n}

\newcommand{\ord}{\mathrm{ord}\,}

\newcommand{\rank}{\mathrm{rank}\,}

\usepackage[all]{xy}
\usepackage[dvips]{graphicx}

\newcommand{\B}{\mathcal{B}}

\title[Non-trivalent graph cocycle and the long knot space]{Non-trivalent graph cocycle and cohomology of the long knot space}
\author{Keiichi Sakai}
\date{\today}
\thanks{The author is supported by 21st century COE program at Graduate School of Mathematical Sciences,
University of Tokyo}
\address{Graduate School of Mathematical Sciences, University of Tokyo}
\email{ksakai@ms.u-tokyo.ac.jp}

\begin{document}

\maketitle

\begin{abstract}
In this paper we show that via the configuration space integral construction a non-trivalent graph cocycle can
also yield a non-zero cohomology class of the space of higher (and even) codimensional long knots.
This simultaneously proves that the Browder operation induced by the operad action defined by R. Budney
is not trivial.
\end{abstract}

\section{Introduction}

Recently the (co)homological properties of the spaces $\K$ (or $\fK$) of (framed) long knots in $\R^n$ are widely
studied; the classical case ($n=3$) by R. Budney \cite{Budney03}, Budney and F. Cohen \cite{BudneyCohen05};
the case $n>3$ by D. Sinha \cite{Sinha02, Sinha04}, V. Turchin \cite{Tourtchine04_2}, P. Salvatore \cite{Salvatore06},
P. Lambrechts, Turchin and I. Voli\'c \cite{LTV06}, and others.
Their approaches in some senses make use of the {\it little disks operad} and its action on $\fK$, which induces on
$H_* (\fK )$ the {\it Browder operation}, a structure of a {\it Poisson algebra}.
This Poisson structure has not been well understood, and studied in \cite{K06, Salvatore06, Tourtchine04_2}, and so on.

There is another geometric approach to $H^*_{DR}(\K )$ (or $H^*_{DR}(\fK )$).
A. Cattaneo, P. Cotta-Ramusino and R. Longoni \cite{CCL02} constructed a cochain map from certain graph complex to the
de Rham complex of $\K$ ($n>3$) via perturbative expansion of Chern-Simons theory, which generalizes the integral
expression of the Vassiliev invariants for knots in $\R^3$ due to R. Bott, C. Taubes \cite{BottTaubes94} and
independently to T. Kohno \cite{Kohno94}.
Moreover they proved that the induced map on cohomology is injective on the trivalent graph cocycles.
The injectivity was proved by evaluating the cohomology classes over the cycles obtained from {\it chord diagrams}.

Almost nothing has been known about the cohomology classes coming from non-trivalent graphs
(in the case of ordinary knots, there is a result of Longoni \cite{Longoni04}; see below).
One reason is that we do not know the corresponding homology cycles.

In this paper we combine the de Rham theory for $\K$ with the action of little disks operad, and obtain the first
example of a non-trivalent graph cocycle which realizes a non-zero cohomology class of $\K$.

\begin{thm}[for the notations, see \S \ref{BottTaubes}]\label{main}
Suppose $n>3$ is odd.
Then the graph cohomology group $H^{3,1}(\D^* )$ consisting of trivalent graphs with exactly one four-valent vertex
is isomorphic to $\R$.
Its generator $\Gamma$ gives a non-trivial element $I(\Gamma ) \in H^{3n-8}_{DR}(\K )$ via the configuration
space integral.
\end{thm}

Theorem \ref{main} is an analogous result to those of \cite{K06, Salvatore06}, but the proof is more geometric.
We prove the non-triviality of $I(\Gamma )$ by evaluating it on a cycle produced by the action of little disks
operad on the space $\fK$, defined in \cite{Budney03}.
Thus we immediately obtain the following.

\begin{cor}\label{Browder_nonzero}
When $n>3$ is odd, the Browder operation induced by the operad action on $\fK$ in the sense of \cite{Budney03}
is non-trivial.\qed
\end{cor}

The cohomology classes of $\K$ obtained from trivalent graphs can be seen as ``higher dimensional analogues'' of
the finite type invariants for knots in $\R^3$.
But the cohomology class obtained in Theorem \ref{main} is not such a one.
It would be an interesting problem to which `invariant' for knots (or 3-manifolds) our class corresponds.

When the first version of this paper was submitted, the author has not been aware of Longoni's result
\cite{Longoni04} for the space $\mathrm{Emb}\, (S^1 ,\R^n )$ of closed (ordinary) knots in $\R^n$.
Longoni found a non-trivalent graph cocycle (different from ours) when $n>3$ is even, and made a non-zero element
of $H^{3(n-3)+1}(\mathrm{Emb}\, (S^1 ,\R^n ))$ from the cocycle.
The proof is also similar to ours, that is, the evaluation of the cocycle on the dual cycle.
But the construction of the cycle naturally differs from ours, since there is no operad action on
$\mathrm{Emb}\, (S^1 ,\R^n )$.
Longoni's cycle is ``secondarily'' defined by using 4-term relations, while we use an operad action.

This paper is organized as follows.
In the second section we recall the action of little disks operad on the space of framed knots, following
\cite{Budney03}, and construct a cycle on which our cocycle will be evaluated.
The third section is devoted to reviewing the configuration space integral.
The readers familiar with \cite{CCL02} may skip this section, except for \S \ref{non_trivalent}.
In the last section we will prove Theorem \ref{main} and give a brief comment on the further computation on
$H^{k(n-3)+1}_{DR}(\K )$, $k\ge 4$.

\subsection*{Acknowledgment}

The author expresses his appreciation to Alberto Cattaneo, Ryan Budney, Fred Cohen, Toshitake Kohno, Paolo Salvatore
and Victor Turchin for reading the draft of this paper and giving him many useful advices.
He is also grateful to Ismar Voli\'c for pointing out a mistake in the previous version of this paper.
The author noticed a result of Longoni by virtue of an information in Dai Tamaki's website (which is written in
Japanese).

\section{The space of long knots and little disks action}

\subsection{The space of long knots}

In this paper we always assume $n>3$ is odd.

\begin{definition}
A {\it long knot} in dimension $n$ is an embedding
\[
 f : \R^1 \hookrightarrow \R^n
\]
such that $f(t) =(0,\dots ,0,t)$ if $\abs{t} \ge 1$.

Denote the unit ball in $\R^m$ by $B^m$;
\[
 B^m :=\{x \in \R^m \, |\, \abs{x} \le 1\} .
\]
A {\it framed long knot} in dimension $n$ is an embedding
\[
 g : B^{n-1}\times \R^1 \hookrightarrow B^{n-1}\times \R^1
\]
such that $g(x,t)=(x,t)$ if $\abs{t} \ge 1$.
Denote the space of all (framed) long knots in $\R^n$ by $\K$ (respectively $\fK$).\qed
\end{definition}

The space $\fK$ defined as above was denoted by $EC(1,B^{n-1})$ in \cite{Budney03}.
We define the framed long knots as in the cylinder, because it becomes easier in this setting to define the little
disks action.

We have a forgetting map $r : \fK \to \K$ defined by
\[
 r(f)(t)=f(0,t),\quad \forall f \in \fK .
\]

\begin{lem}[\cite{Budney03}]
The map $r$ is equivalent to a trivial fibration with fiber $\Omega SO(n-1)$.
Hence $\fK \simeq \K \times \Omega SO(n-1)$.\qed
\end{lem}

\subsection{Some cycles}\label{cycles}

The subgroup $\bigoplus_{k\ge 0}H_{(n-3)k}(\K )$ is known to be non-trivial
\cite{CCL02, Longoni04, Sinha04, Tourtchine04_2} since it contains the subalgebra isomorphic to the algebra $\A$
of {\it chord diagrams} modulo {\it 4-term} and {\it 1-term} relations, and few other cycles are known.
The purpose of this paper is to find another (co)cycle which does not come from any chord diagrams.

Here we explain two examples of cycles made from chord diagrams, namely, $e \in H_{n-3}(\fK )$ and
$v_2 \in H_{2(n-3)}(\K )$ which we will use later.
For more general treatment, see \cite{CCL02, Longoni04, K06, Tourtchine04_2}.

\subsubsection{The cycle $v_2$}
Consider the chord diagram $V$ in Figures \ref{chord_diagram}, which is thought of as corresponding to an immersion
$f$ with two transversal doublepoints $z_i = f(\xi_i )=f(\xi_{i+2})$, $i=1,2$, $\xi_1 < \xi_2 < \xi_3 < \xi_4$,
see Figure \ref{blowup}.
Since we assume $n>3$, the immersion $f$ is determined uniquely up to homotopy.

\begin{figure}[hbt]
\begin{center}
\includegraphics[width=12cm, height=2cm]{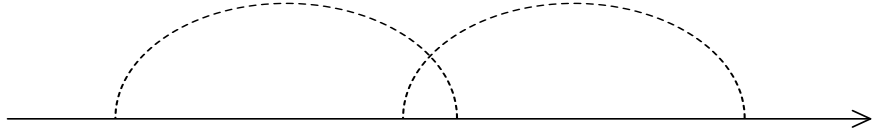}
\end{center}
\caption{Chord diagram $V$}
\label{chord_diagram}
\end{figure}

At each self-intersections $z_i$ we have resolutions of $f$ parametrized by $S^{n-3}$ (see Figure \ref{resolve}),
so we obtain a map
\[
 \alpha (V) : (S^{n-3})^2 \longrightarrow \K .
\]
More explicitly, the knot $\alpha (V)(u_1 ,u_2 )$ is defined in \cite{CCL02, Longoni04} by
\[
 \alpha (V)(u_1 ,u_2 ) (t) =
 \begin{cases}
  f(t)+\delta_i u_i \exp \left[ \frac{1}{(t-\xi_i )^2 -\varepsilon^2_i}\right]
   & \abs{t-\xi_i} < \varepsilon_i ,\ i=1,2 \\
  f(t) & \text{otherwise}
 \end{cases}
\]
where $\delta_i$ and $\varepsilon_i$ are small positive numbers, and $u_i \in S^{n-3}$ is realized as a unit vector
in $\R^n$ which is perpendicular to $f' (\xi_i )$ and $f' (\xi_{i+2})$.

\begin{figure}[hbt]
\begin{center}
\includegraphics[width=12cm,height=3cm]{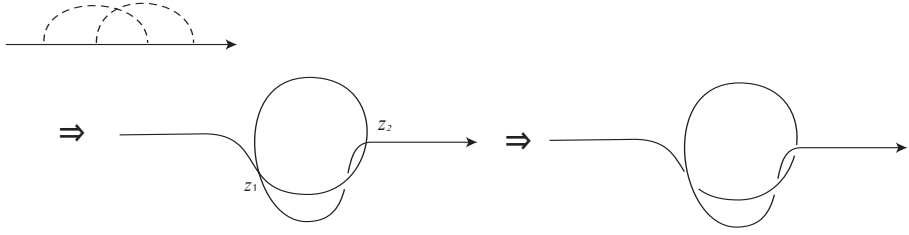}
\end{center}
\caption{The correspondence $\Gamma \mapsto f_{\Gamma} \mapsto \alpha (\Gamma )$}\label{blowup}
\end{figure}

\begin{figure}[htb]
\begin{center}
\includegraphics[width=12cm, height=3cm]{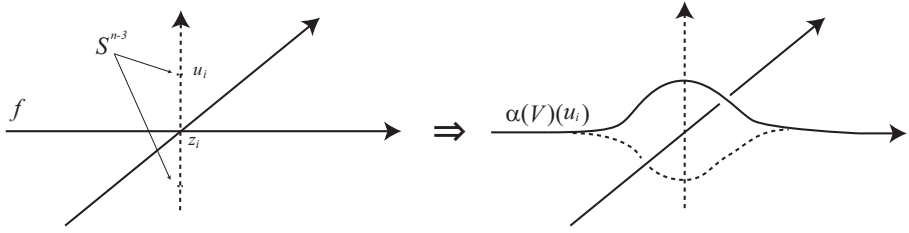}
\end{center}
\caption{Resolution of the self-intersection}\label{resolve}
\end{figure}

\begin{rem}[\cite{CCL02}]\label{linkingnumber}
The union of all the resolutions generate $S^{n-2}$;
\[
 S_i := \bigcup_{u_i \in S^{n-3}} \bigcup_{\abs{t -\xi_i} < \varepsilon_i} \alpha (V)(u_1 ,u_2)(t)\approx S^{n-2},
\]
which has the linking number one with the segment $I_i :=f(\xi_{k+i} -\varepsilon_i ,\xi_{k+i} +\varepsilon_i )$
(see Figure \ref{resolve}).\qed
\end{rem}

We regard the map $\alpha (V)$ as a $2(n-3)$-cycle of $\K$, and denote its homology class by $v_2 \in H_{2(n-3)}(\K )$
(because it can be seen as a dual of the order two invariant for knots in $\R^3$; see \cite{BottTaubes94, Kohno94}).

This construction extends to general chord diagrams.

\begin{prop}[\cite{CCL02, Longoni04, K06, Tourtchine04_2}]
The correspondence $\Gamma \mapsto \alpha (\Gamma )$ is defined for any chord diagrams and determines an injective
homomorphism of algebras
\[
 \alpha : \A \longrightarrow \bigoplus_{k\ge 0}H_{(n-3)k}(\K ),
\]
where $\A$ is an algebra generated by chord diagrams modulo {\it 4-term relations} (see Figure \ref{4t}) and
{\it 1-term relation}, that is, a chord diagram with an {\it isolated chord} (a chord which does not intersect
with other chords) is regarded as zero.\qed
\end{prop}

\begin{figure}[hbt]
\begin{center}
\includegraphics[width=12cm, height=5cm]{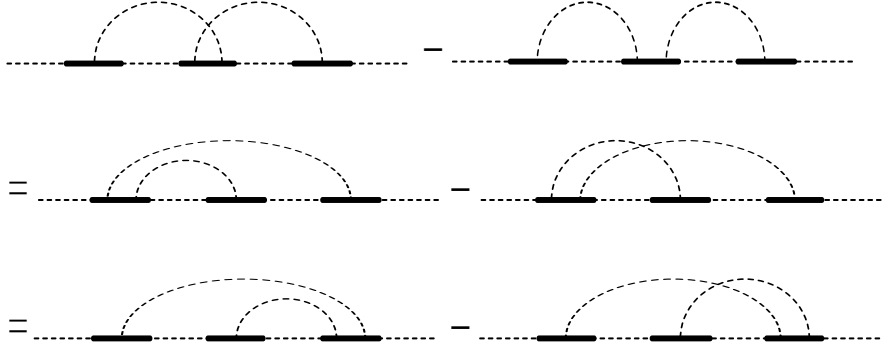}
\end{center}
\caption{4-term relations}
\label{4t}
\end{figure}

\subsubsection{The cycle $e$}
A resolution of an isolated chord yields a null-homologous cycle of $\K$
(recall the Reidemeister move I for knots in $\R^3$).
Instead we assign a homology cycle $e \in H_{n-3}(\Omega SO(n-1))$ to $\Gamma_0$, a chord diagram with only one chord.
First consider the cycle $e'$ of $SO(n-1)$ realized by a map
\[
 e' : \Sigma S^{n-3} \longrightarrow SO(n-1),
\]
called ``the clutching map for the tangent bundle of $S^{n-1}$,'' defined below;
where the suspension $\Sigma S^{n-3} \approx S^{n-2}$ is defined by collapsing the subsets
$(-\infty ,-1] \times S^{n-3}$ and $[1,\infty )\times S^{n-3}$ of $\R^1 \times S^{n-3}$ to points.
We think of $S^{n-3}$ as a unit sphere in $\R^{n-2}$ and $\Sigma S^{n-3}$ as in $\R^{n-1}$, by using an inclusion
$\Sigma S^{n-3} \to \R^{n-2} \times \R^1$,
\[
 [s,u] \longmapsto
  \begin{cases}
   (\sqrt{1-s^2 }\, u,s) & s \in [-1,1], \\
   (0,-1) & s \le -1, \\
   (0,1) & s \ge 1.
  \end{cases}
\]
The map $e' : \Sigma S^{n-3} \to SO(n-1)$ is defined by
\[
 e'[s,u]=H_{x_{n-1}}H_{[s,u]} ,
\]
where $H_{[s,u]} \in O(n-1)$ is the reflection with respect to the orthogonal complement of $[s,u] \in \R^{n-1}$
($H_{x_{n-1}}$ is the reflection with respect to the hyperplane $\{ x_{n-1}=0\}$).
Since $e'(1,u)=e'(-1,u)=I_{n-1}$, the adjoint map
\[
 e : S^{n-3} \longrightarrow \Omega SO(n-1), \quad e(u)(s) := e' [s,u]
\]
to $e'$ is defined and represents the desired cycle $[e]$.
It is known that $[e]$ determines non-trivial homology class only if $n$ is odd.

We regard $e \in H_{n-3}(\fK )$ by composing $e$ with $j:\Omega SO(n-1) \hookrightarrow \fK$ defined by
\[
 j(\gamma )(x,t)= (\gamma (t)x,t),
\]
here $\gamma (t) \in SO(n-1)$ is seen as a linear transformation in $\R^{n-1} \times \{ 0\}$.

We have described two cycles $e$ and $v_2$.
Below we will show that the Poisson bracket $\lambda (e, v_2 )$ is not zero.
The Poisson structure is induced from an action of little disks operad, and will be explained in the next subsection.

\subsection{Little disks action}\label{action}

\begin{definition}
A {\it little $m$-ball} is an embedding $b : B^m \hookrightarrow B^m$ of the form
\[
 b(x)=r(x-p)
\]
for some $p \in B^m$ and $0 < r \le 1$. Define the {\it little $m$-balls operad} $\B_m$ by setting
\[
 \B_m (k) := \left\{ (b_1 ,\dots ,b_k ) \, \left| \, {b_i \text{ a little }m\text{-ball}, \atop
 b_i (\mathrm{Int}\, B^m ) \cap b_j (\mathrm{Int}\, B^m ) = \emptyset \text{ if }i \ne j} \right. \right\}
\]
for $k \ge 1$.
The operad structure is defined in a familiar way (see \cite{May271}).\qed
\end{definition}

Here we recall the operad action of $\B_2$ on $\fK$ defined in \cite{Budney03}, that is, the ``associative'' maps
\[
 \kappa (k) : \B_2 (k) \times (\fK )^k \longrightarrow \fK ,\quad k \ge 1.
\]

Given $b=(b_1 ,\dots ,b_k ) \in \B_2 (k)$, consider the projections
\[
 I_j := pr_1 \circ b_j (B^2 ) \subset [-1,1],\quad 1 \le j \le k.
\]
There are the little 1-balls $l_j (t)=a_j t+b_j$ such that $l_j ([-1,1])=I_j$, $1 \le j \le k$
($l_1 ,\dots ,l_k$ are not necessarily disjoint mutually).

A little 1-ball $l:[-1,1] \to [-1,1]$, $l(t)=at+b$, extends to a diffeomorphism $\tilde{l} : \R^1 \to \R^1$
in an obvious way, and determines a map
\begin{gather*}
 \mu_l : \fK \longrightarrow \fK , \\
 \mu_l (f) := (\id_{B^{n-1}}\times \tilde{l}) \circ f\circ (\id_{B^{n-1}}\times \tilde{l}^{-1}).
\end{gather*}

For any little $2$-ball $b$, define the number $t_b \in [-1,1]$ by
\[
 t_b =\mathrm{min}\, \{ y \, | \, (x,y) \in b(B^2 ) \text{ for some }x \} .
\]
With these notations in hand, we can define the map $\kappa$ by
\[
 \kappa (k) ((b_1 ,\dots ,b_k );(f_1 ,\dots ,f_k )):=
 \mu_{l_{\sigma (1)}} (f_{\sigma (1)})\circ \dots \circ \mu_{l_{\sigma (k)}} (f_{\sigma (k)}),
\]
where $\sigma \in \mathfrak{S}_k$ is such that $t_{b_{\sigma (1)}} \le \dots \le t_{b_{\sigma (k)}}$.

\begin{thm}[\cite{Budney03}]
The maps $\kappa (k)$ ($k \ge 1$) are well defined and defines the action of the operad $\B_2$ on $\fK$.\qed
\end{thm}

In particular, $\kappa (2): \B_2 (2) \times (\fK )^2 \to \fK$ is `pushing one long knot $f_1$ through another
long knot $f_2$, afterward pushing $f_2$ through $f_1$' (see Figure \ref{through} and Figures 2, 5, 7 in
\cite{Budney03}).

\begin{figure}[htb]
\begin{center}
\includegraphics{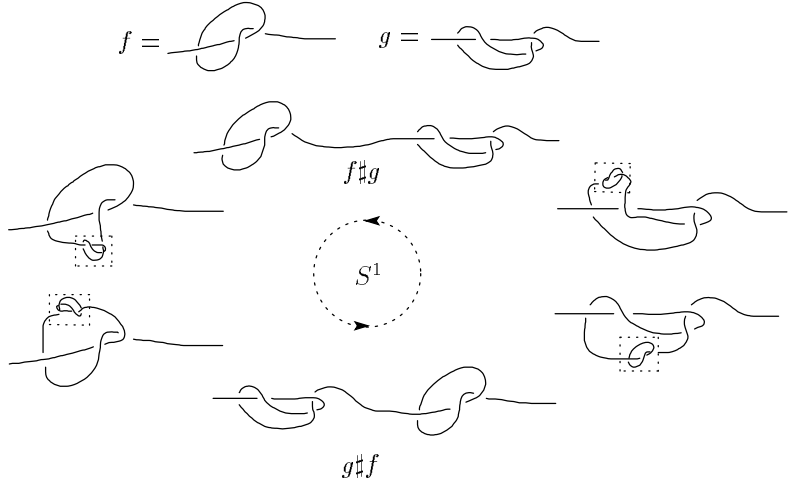}
\end{center}
\caption{A picture of $\kappa (2)$; notice that $\B_2 (2) \simeq S^1$}\label{through}
\end{figure}

The space $\B_2 (2)$ is homotopy equivalent to $S^1$.
The map $\kappa (2) : \B_2 (2) \times (\fK )^2 \to \fK$ induces on homology two products
\begin{align*}
 * : H_p (\fK ) \otimes H_q (\fK ) &\longrightarrow H_{p+q}(\fK ), \\
 \lambda : H_p (\fK ) \otimes H_q (\fK ) &\longrightarrow H_{p+q+1}(\fK )
\end{align*}
corresponding to generators of $H_p (\B_2 (2))$, $p=0,1$, respectively.
The former product is equal to that induced by the connecting sum.
The latter $\lambda$ is called {\it Browder operation} and is a {\it 1-Poisson bracket}, that is, a Lie bracket
of degree one, satisfying the Leibniz rule (see \cite{Cohen533}).

Our attention will be paid to the element $\lambda (e ,v_2 ) \in H_{3n-8}(\fK )$ or its image
$\Lambda := r_* \lambda (e ,v_2 ) \in H_{3n-8}(\K )$ via the forgetful map $r:\fK \to \K$.
For definiteness, we choose a map
\[
 v : (S^{n-3})^2 \longrightarrow \K
\]
representing $v_2$  by resolving an immersion $f$ (Figure \ref{Lambda}).
Most part of the embedding lies in the $x_{n-1}x_n$-plane.
The self-intersections to be resolved are $z_i =f(\xi_i )=f(\xi_{i+2})$, $\xi_i < \xi_{i+2}$, $i=1,2$.
The vectors $u_i \in S^{n-3}$, $i=1,2$ (which are normal to $x_{n-1}x_n$-plane) produce the resolutions of the
self-intersections $z_i$, respectively.
The segments $l$ are included in the $x_n$-axis.

Given the `trivial frame,' $v$ can represent the cycle $[v]=v_2 \in H_{2(n-3)}(\fK )$.

\begin{figure}[htb]
\begin{center}
\includegraphics{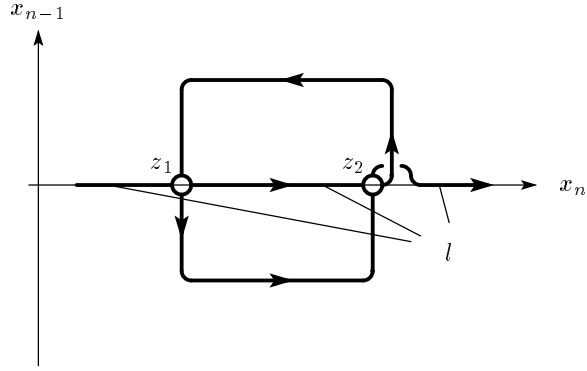}
\end{center}
\caption{The cycle $v_2$}\label{Lambda}
\end{figure}

Then the representative $\lambda (e,v)$ is the family of embeddings defined by `pushing $e$ through $v$,
afterward pushing $v$ through $e$.'

We want to know the representative of $\Lambda  =r_* (\lambda (e,v_2 )) \in H_{3n-8}(\K )$.
$\Lambda$ is obtained from $\lambda (e,v)$ by forgetting the frame.
When $e$ goes through $v$, the frame of $v$ `near' $e$ would be agitated.
But this phenomenon disappears after forgetting the frame via $r:\fK \to \K$.
In contrast, when $v$ passes through $e$, the whole embedding $v$ `rotates' around $x_n$-axis via the frame $e$,
and this phenomenon does not disappear even if we forget the frame.
Thus $\Lambda$ is represented by `$v$ rotated by $e$.'

More precisely, if we think of $SO(n-1)$ as a subgroup of $SO(n)$ fixing the $x_n$-axis, then $\Lambda$ can be
represented by the map
\begin{gather*}
 \Lambda : \Sigma S^{n-3} \times (S^{n-3})^2 \longrightarrow \K , \\
 \Lambda ([s,u_0] ,u_1 ,u_2 )(t) := e( [ 2s+p(v(u_1 ,u_2 )(t)),u_0 ])v(u_1 ,u_2 )(t),
 \quad t \in \R^1 ,
\end{gather*}
where $p : \R^n \to \R^1$ is the projection $(x_1 ,\dots ,x_n )\mapsto x_n$.
Thus $\Lambda ([s,u_0] ,u_1 ,u_2 )$ is a long knot $v(u_1 ,u_2 )$ with its intersection with
$p^{-1}(a)$ being rotated around the $x_n$-axis by the frame $e[2s+a ,u_0 ]\in SO(n-1)$, $\abs{a} \le 1$.

The cycle $\Lambda$ has a simpler description; for $0 \le \tau \le 1$, define
\begin{gather*}
 \Lambda'_{\tau} : \Sigma S^{n-3} \times (S^{n-3})^2 \longrightarrow \K , \\
 \Lambda'_{\tau} ([s,u_0] ,u_1 ,u_2 )(t) := e( [ (2-\tau )s+(1-\tau )p(v(u_1 ,u_2 )(t)),u_0 ]) v(u_1 ,u_2 )(t),
\end{gather*}
then $\Lambda'_{\tau}$ is well-defined for any $\tau \in [0,1]$, $\Lambda_0' =\Lambda$ and
\[
 \Lambda'_1 ([s,u_0] ,u_1 ,u_2 )(t) = e( [s, u_0 ])v(u_1 ,u_2 )(t).
\]
Below we rewrite $\Lambda := [\Lambda'_1 ] \in H_{3n-8}(\K )$.
This $\Lambda$ is $v(u_1 ,u_2 )$ rotated all together by $e[s,u_0 ]$.

\begin{rem}
There are several ways to define the action of $\B_2$. In \cite{Sinha04} D. Sinha constructed a cosimplicial model
for the space $\K'$ of `long knots modulo immersions,' a space which relates to $\fK$, and proved that the space is
a little disks object by means of McClure-Smith machinery \cite{McClureSmith04}.
It can be proved \cite{K06, Salvatore06} that, when $n>3$ is odd, the induced Browder operation is not zero;
\[
 \lambda : H_{n-3} (\K' ,\R ) \otimes H_{2(n-3)} (\K' ,\R ) \xrightarrow{\cong} H_{3n-8}(\K' ,\R ).
\]
It is still unknown how the operad actions on $\fK$ and $\K'$ relate to each other.
So Corollary \ref{Browder_nonzero} is the first result about the non-triviality of the Browder operation in the sense
of \cite{Budney03}.\qed
\end{rem}

\section{Configuration space integral}\label{BottTaubes}

Here we recall the main result of \cite{CCL02} when $n>3$ is odd.
For even dimensional case see \cite{CCL02}.
Readers also refer to \cite{BottTaubes94, Kohno94, Volic05}.

\subsection{Graph complex}

\begin{definition}[\cite{CCL02}]
Our {\it graph} consists of the following data.

\begin{enumerate}
\item Any graph has an oriented line called the special line.

\item A graph has two types of vertices (the set of vertices is possibly empty);
those on the special line and those not on the line.
In \cite{Volic05} the former vertices are called {\it interval} ones, while the latter {\it free}.
The vertices are labeled by $1,2,\dots ,m$ for an appropriate $m\ge 0$ so that the labels of the interval vertices
are smaller than those of free vertices.

\item Vertices are connected by oriented edges so that the graph is connected.
The valency of each vertex is at least three.
An edge may have only one interval vertex as its end-points (such an edge is called a {\it small loop}).

\item If an edge $e$ is a small loop at the interval vertex, then we give the order of the half-edges of $e$
(which is defined independently of the orientation of $e$).

\end{enumerate}
Let $\Gamma$ be a graph with $e$ edges, $v_i$ interval vertices and $v_f$ free vertices.
Define
\begin{align*}
 \ord \Gamma &:= e-v_i, \\
 \deg \Gamma &:= 2e -3v_f -v_i .\qed
\end{align*}
\end{definition}

An example of a graph is shown in Figure \ref{ex_graph}.

\begin{figure}[htb]
\begin{center}
\includegraphics[width=12cm, height=2cm]{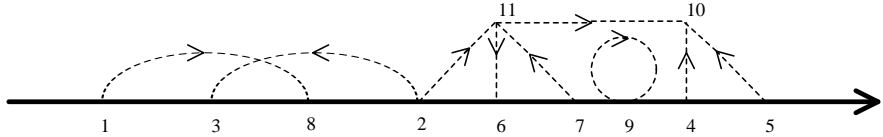}
\end{center}
\caption{An element of $\D^{7,3}$}\label{ex_graph}
\end{figure}

\begin{rem}
For any graph $\Gamma$, its order and degree are not less than zero.
One can easily prove that $1-\ord \Gamma$ is equal to the Euler characteristic of the one dimensional CW-complex
$\Gamma$, and that $\deg \Gamma$ is zero if and only if $\Gamma$ is a trivalent graph.\qed
\end{rem}

Consider the vector space spanned by the graphs with $\ord \Gamma =k$ and $\deg \Gamma =l$ modulo the subspace
generated by
\begin{enumerate}
\item $\Gamma$, two vertices of which are joined by more than one edges,
\item $\Gamma$ with a small loop whose endpoint is not a free vertex, and
\item $\Gamma^{\prime}-(-1)^{\mathrm{sign}\, \sigma}\Gamma$, here $\Gamma^{\prime}$ is obtained from $\Gamma$
by a permutation $\sigma$ which permutes the labels of the vertices (so that the labels of the interval vertices are
less than those of free vertices) or reversing the orientations of the edges.
\end{enumerate}
We denote the quotient space by $\D^{k,l}$.

The differential $\delta :\D^{k,l}\to \D^{k,l+1}$ is defined as follows.
For any graph $\Gamma$, $\delta \Gamma$ is the signed sum of graphs obtained by contracting, one at a time,
the edges one of whose endpoint is not interval, and the {\it arcs}, portions of the special line
bounded by two consecutive interval vertices.

Determining the labels and signs of the graphs after contraction (see \cite{CCL02}), we can show the following
directly by definition.

\begin{thm}[\cite{CCL02}]
The map $\delta$ sends $\D^{k,l}$ to $\D^{k,l+1}$, and $\delta^2 =0$.\qed
\end{thm}

\begin{ex}\label{ex_trivalent_cocycle}
Two examples of $\delta : \D^{2,0} \to \D^{2,1}$ are given in Figure \ref{cocycle_v2}.

\begin{figure}[htb]
\begin{center}
\includegraphics[width=12cm, height=2cm]{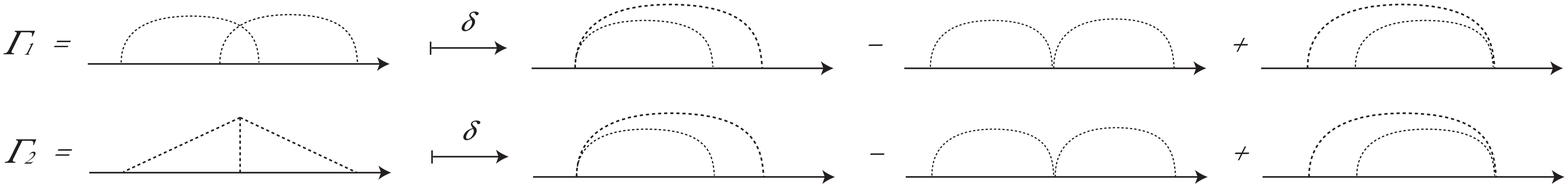}
\end{center}
\caption{Examples of the coboundary operator}\label{cocycle_v2}
\end{figure}

A cochain $\Gamma_1 -\Gamma_2$ is thus a cocycle in $\D^{2,0}$.\qed
\end{ex}

\subsection{Configuration space integrals}

Below we will associate a differential form of $\K$ with a given graph $\Gamma$.

We denote the {\it configuration space} by
\[
 \Conf (X,m) := \{ (x_1 ,\dots ,x_m ) \in X^m \, |\, x_i \ne x_j \} .
\]
For any $N$, the {\it Gauss maps} $\varphi_{ij}:\Conf (\R^N ,m)\to S^{N-1}$ ($1\le i\ne j\le m$) are defined by
\[
 \varphi_{ij}(x_1 ,\dots ,x_m )=\frac{x_i -x_j}{\abs{x_i -x_j}}.
\]
We use the compactifications of the configuration spaces.

\begin{thm}[\cite{AxelrodSinger94, BottTaubes94, Sinha03}]
For any manifold $M$, we can construct a compact manifold $\Conf [M,m]$ with corners, which is a compactification of
$\Conf (M,m)$ in the sense that the interior of $\Conf [M,m]$ is $\Conf (M,m)$.
When $M=\R^N$, then the Gauss maps $\varphi_{ij}$ can be extended smoothly onto the boundary of $\Conf [\R^N ,m]$.\qed
\end{thm}

Roughly speaking, the points in $\Conf [M,m]$ may `collide with each other,' but in such cases,
information of the directions of the collision must be recorded.

Let $\Gamma \in \D^{k,l}$ be a graph with $e$ edges, $v_i$ interval vertices and $v_f$ free vertices
(hence $e-v_f=k$, $2e-3v_f -v_i =l$).
Consider the following pull-back square:
\[
\xymatrix{
 \Conf [\R^n ;v_i ,v_f ]  \ar[d]^-p \ar[r]^-{\tilde{ev}}
  & \Conf[\R^n ,v_i +v_f ] \ar[d]^-{\pi_{v_l}} \ar[r]^-{\varphi_{ij}} & S^{n-1} \\
 \mathrm{Conf}_0 [\R^1 ,v_i]\times \K \ar[r]^-{ev} & \Conf[\R^n ,v_i ] &
}
\]
here $\mathrm{Conf}_0 [\R^1 ,m]$ is a connected component corresponding to $t_1 \le t_2 \le \dots \le t_m$ and,
on the interior, $ev$ and $\pi_*$ are defined by
\begin{align*}
 ev((t_1 ,\dots ,t_{v_i}),f) &:= (f(t_1 ),\dots ,f(t_{v_i})), \\
 \pi_{v_i}(x_1 ,\dots ,x_{v_i +v_f}) &:= (x_1 ,\dots ,x_{v_i}).
\end{align*}
$\Conf [\R^n ;v_i ,v_f ]$ is the space of pairs $((x_1 ,\dots ,x_{v_i +v_f}),f)$, where
\[
 (x_1 ,\dots ,x_{v_i +v_f})\in \Conf [\R^n ,v_i +v_f ]
\]
and $f\in \K$, with $x_1 ,\dots ,x_{v_i}$ on the knot $f$.

With an edge (or a small loop) $\overrightarrow{ij}$ of $\Gamma$, we assign a differential form
$\theta_{ij}\in \Omega^{n-1}(\Conf [\R^n ;v_i ,v_f])$ defined by
\[
 \theta_{ij} :=
  \begin{cases}
    \tilde{ev}^*\varphi^*_{ij}vol_{S^{n-1}} & i \ne j, \\
    D_i^* vol_{S^{n-1}} & i=j,
  \end{cases}
\]
here, for $1\le i\le v_i$,
\begin{gather*}
 D_i :\Conf [B^{n-1}\times \R^1 ;v_i ,v_f] \longrightarrow S^{n-1}, \\
 D_i (f(t_1 ),\dots ,f(t_{v_i}),x_1 ,\dots ,x_{v_f}) =\frac{f^{\prime}(t_i )}{\abs{f^{\prime}(t_i )}}.
\end{gather*}
We define a form $\theta_{\Gamma} \in \Omega^{(n-1)e}(\Conf [\R^n ;v_i ,v_f ])$ by
\[
 \theta_{\Gamma}:=\bigwedge_{\text{edges }\overrightarrow{ij}\text{ of }\Gamma}
 \theta_{ij}.
\]
Note that this form is defined independently of the order of $\theta_{ij}$'s, since they are even forms.

Integrating this form along the fiber of
\[
 \Conf [\R^n ;v_i ,v_f ]\stackrel{p}{\longrightarrow} \mathrm{Conf}_0 [\R^1 ,v_i]\times \K
 \stackrel{pr_2}{\longrightarrow} \K,
\]
we obtain a differential form
\[
 I(\Gamma ):=(pr_2 \circ p)_* \theta_{\Gamma}\in \Omega^* (\K ).
\]
This integral actually converges since we compactify the configuration spaces.
The degree of the form $I(\Gamma )$ is
\begin{align*}
 (n-1)e-nv_f -v_i &= (n-3)(e-v_f )+2e-3v_f -v_i  \\
 &= (n-3)k+l.
\end{align*}
Thus we have a map
\[
 I:\D^{k,l}\to \Omega^{(n-3)k+l}(\K ).
\]

\begin{thm}[\cite{CCL02}]\label{cochain}
If $n>3$ is odd, then the above map $I$ is a cochain map.
\end{thm}

\begin{proof}[Outline of proof]
By Stokes' theorem, the differential $dI(\Gamma )$ is an integration along the boundary of
$\Conf [\R^n ;v_i ,v_f]$.
Recall \cite{AxelrodSinger94} that the boundary of $\Conf [\R^n ;v_i ,v_f ]$ is stratified via the
`complexities of degenerations of the configurations.'
The codimension one strata correspond to the simultaneous collisions of points.
We can see \cite{BottTaubes94, CCL02, Volic05} that, when $n>3$, only the `principal faces' (corresponding to the
collisions of exactly two points) contribute to the integration $dI(\Gamma )$.
These collisions exactly correspond to the differential $\delta$ of the graph complex $\D^*$, hence
$dI(\Gamma )=I(\delta \Gamma )$.
\end{proof}

\subsection{Trivalent graph cocycles}\label{ex_evaluation}

Notice that a chord diagram with $k$ chords is thought of as in $\D^{k,0}$ if some orientation of edges are given
to $\Gamma$.
The chord diagram has $e=k$ edges, $v_i =2k$ interval vertices, and no free vertices ($v_f =0$).

The proof of the following is a combinatorial one.

\begin{lem}[\cite{CCL02}]
Let $\Gamma =\sum_i a_i \Gamma_i \in \D^{k,0}$ be a non-zero cocycle with each $\Gamma_i$ trivalent graphs.
Then there is at least one graph, say $\Gamma_1$, which is a chord diagram.
Moreover, all the chord diagrams contained in the summand of $\Gamma$ has no isolated chord.\qed
\end{lem}

For example, a cochain $\Gamma_1 -\Gamma_2$ given in Example \ref{ex_trivalent_cocycle} contains a chord diagram
$\Gamma_1$, and there is no isolated chord.

Let $\Gamma =\sum_i a_i \Gamma_i \in \D^{k,0}$ be a non-zero cocycle, and suppose $\Gamma_1$ is a chord diagram with
$a_1 \ne 0$.
Then $I(\Gamma ) \in H^{(n-3)k}_{DR}(\K )$ turns out to be not zero by the following theorem.

\begin{thm}[\cite{CCL02}]\label{CCL}
Denote by $\langle \, , \, \rangle$ the pairing of cocycles with cycles.
Then we have $\ang{I(\Gamma )}{\alpha (\Gamma_1 )} =\pm a_1$.\qed
\end{thm}

A detailed proof can be found in \cite{CCL02}.
Here, as an example, we compute $\ang{I(\Gamma )}{v_2}$ where $\Gamma =\Gamma_1 -\Gamma_2$ is a cocycle given in
Example \ref{ex_trivalent_cocycle} (notice that the cycle $v_2$ is equal to $\alpha (\Gamma_1 )$).
This computation is easily generalized to prove Theorem \ref{CCL}, and gives us a lot of useful suggestions for
the proof of our main theorem.

Let $v_i$ and $v_f$ the numbers of interval and free vertices of the graph $\Gamma_j$, $j=1,2$
(if $j=1$, then $v_i =4$ and $v_f =0$; if $j=2$, then $v_i =3$, $v_f =1$).
Consider the following pull-back square;
\[
 \xymatrix{
 (\id \times \alpha (\Gamma_1 ))^* \Conf [\R^n ;v_i ,v_f ] \ar[rr]^-{\beta} \ar[d]_-{\tilde{p}} & &
  \Conf [\R^n ;v_i ,v_f ] \ar[d]^-p \\
 \mathrm{Conf}_0 [\R^1 ,v_i ]\times (S^{n-3})^2 \ar[rr]^-{\id \times \alpha (\Gamma_1 )} \ar[d]_-{pr_2} & &
  \mathrm{Conf}_0 [\R^1 ,v_i ]\times \K \ar[d]^-{pr_2} \\
 (S^{n-3})^2 \ar[rr]^-{\alpha (\Gamma_1 )} & & \K
 }
\]
Then
\[
 \ang{I(\Gamma_j )}{\alpha (\Gamma_1 )} = \int_{(S^{n-3})^2} (pr_2 \circ \tilde{p})_* \beta^* \theta_{\Gamma_j}
 = \int_{\mathrm{Conf}_0 [\R^1 ,v_i ]\times (S^{n-3})^2} \tilde{p}_* \beta^* \theta_{\Gamma_j},
\]
and in this case the integrands are
\[
 \theta_{\Gamma_1} = \theta_{13}\theta_{24}, \quad \theta_{\Gamma_2} = \theta_{14}\theta_{24}\theta_{34}.
\]
Recall that the immersion $f$ has the transversal self-intersections
$z_i =f(\xi_i )=f(\xi_{i+2})$, $\xi_1 < \xi_2 < \xi_3 < \xi_4$.
Let $\varepsilon_i >0$ ($i=1,2$) be sufficiently small numbers appeared in the definition of the resolution of
$f$ (\S \ref{cycles}) and define the subspace $C=C_{\varepsilon} \subset \mathrm{Conf}_0 [\R^1 ,v_i ]$ by
\[
 C_{\epsilon} := \left\{ (t_1 ,\dots ,t_{v_l}) \in \mathrm{Conf}_0 [\R^1 ,v_i ] \, \left| \,
 {1 \le \exists i \le 4,\ \abs{t_m -\xi_i}>\varepsilon_i \atop \text{for any }1 \le m \le v_i} \right. \right\} .
\]
Notice that the complementary set $\mathrm{Conf}_0 [\R^1 ,v_i ] \setminus C$ is the set of configurations
such that there is at least one $t_m$ near $\xi_i$, $1 \le i \le 4$.

Write $\omega^{(j)} := \tilde{p}_* \beta^* \theta_{\Gamma_j}$ and
\[
 \omega^{(j)}_C := \int_C \omega^{(j)} , \quad
 \eta^{(j)}_C := \int_{\mathrm{Conf}_0 [\R^1 ,v_i ] \setminus C} \omega^{(j)} .
\]
Then
\[
 \ang{I(\Gamma_j )}{\alpha (\Gamma_1 )} = \int_{(S^{n-3})^2} \omega^{(j)}_C +\int_{(S^{n-3})^2} \eta^{(j)}_C .
\]
Even if we reduce the `sizes' $\delta_i$ of the resolutions of $i$-th self-intersection of the immersion $f$
(see \S \ref{cycles}), we still have a homologous cycle $\alpha (\Gamma_1 )$, hence the value
$\ang{I(\Gamma )}{\alpha (\Gamma_1 )}$ remains unchanged.
So we have
\[
 \ang{I(\Gamma )}{v_2} = \lim_{\delta_1 ,\delta_2 \to 0}
 \sum_{j=1,2}\left( \int_{(S^{n-3})^2} \omega^{(j)}_C +\int_{(S^{n-3})^2} \eta^{(j)}_C \right) .
\]
But the limit of the integration of $\omega^{(j)}_C$ is zero, since on $C$ there is at least one $\xi_i$ whose
neighborhood does not contain any configuration point $t_m$, then the size of the resolution can be reduced to
exactly zero at the corresponding doublepoint $z_k$ (because collision of configuration points never occur),
and the dimension of the cycle decreases.

Thus only the second term, the integration over $\mathrm{Conf}_0 [\R^1 ,v_i ] \setminus C$ contributes to the limit
of $\ang{I(\Gamma )}{\alpha (\Gamma_1 )}$.
Since there are four $\xi_i$'s, $\mathrm{Conf}_0 [\R^1 ,v_i ] \setminus C \ne \emptyset$ only if $v_i \ge 4$.
But $\Gamma_2$ has only three interval vertices, so cannot contribute to the pairing, while $\Gamma_1$ may
contribute to the pairing since it has four interval vertices
(in general cases, $\mathrm{Conf}_0 [\R^1 ,v_i ] \setminus C \ne \emptyset$ only for the graphs which are chord
diagrams).

So it suffices to compute the limit of
\begin{align*}
 \int_{(S^{n-3})^2} \eta^{(1)}_C
 &= \int_{(\mathrm{Conf}_0 [\R^1 ,4]\setminus C) \times (S^{n-3})^2} \tilde{p}_* \beta^* \theta_{\Gamma_1} \\
 &= \int_{\beta \tilde{p}^{-1}\{ (\mathrm{Conf}_0 [\R^1 ,4]\setminus C) \times (S^{n-3})^2 \}} \theta_{13}\theta_{24}.
\end{align*}
Recall from Remark \ref{linkingnumber} the $(n-2)$-spehre $S_i$ generated by all the resolution of $z_i$,
which has the linking number one with the segment $I_i$ ($i=1,2$).
The set
\[
 \beta \tilde{p}^{-1}\{ (\mathrm{Conf}_0 [\R^1 ,4]\setminus C) \times (S^{n-3})^2 \}
\]
is precisely the disjoint union $\bigsqcup_{i=1,2}S_i \sqcup I_i$, and the above integration is
\[
 \prod_{i=1,2} \int_{S_i \times I_i} \varphi_i^* vol_{S^{n-1}}
\]
for the Gauss map $\varphi_i : S_i \times I_i \to S^{n-1}$.
Its limit is the product of the linking numbers of $S_i$ and $I_i$ ($i=1,2$), thus equal to one.
Thus
\[
 \ang{I(\Gamma )}{\alpha (\Gamma_1 )}
 = \lim_{\delta_i \to 0}\ang{I(\Gamma_1 )}{\alpha (\Gamma_1 )} = \pm 1 .\qed
\]

\subsection{Non-trivalent graph cocycle}\label{non_trivalent}

At present it is not known in general whether the map $I : H^{k,l}(\D^* ) \to H^{(n-3)k+l}_{DR}(\K )$, $l>0$,
yields non-trivial cohomology class of $\K$.
But we can see \cite{K06, Salvatore06, Tourtchine04_2} that, when $n>3$ is odd,
\[
 \rank H_{3n-8}(\K ) = 1.
\]
So we can expect that $I : H^{3,1}(\D^* )\to H^{3n-8}_{DR}(\K )$ might produce a non-trivial cohomology class
which is dual to the generator of $H_{3n-8}(\K ,\R )$.

It is difficult to compute $H^{3,l}(\D^* )$ ($l \ge 1$) by hand, but computer calculus tells us the following.

\begin{lem}
If $n$ is odd, then $H^{3,1}(\D^* )\cong \R$.
As a generator we can choose the cochain shown as in Figure \ref{cocycle}.\qed
\end{lem}

It can be easily seen by a direct computation that the cochain $\Gamma$ in Figure \ref{cocycle} is really a cocycle.
It cannot be a coboundary, since $I(\Gamma ) \in H^{3n-8}(\K )$ is not zero as we will prove later.

In Figure \ref{cocycle}, we omit the labels of the vertices and the orientations of the edges.
Unless otherwise indicated,
\begin{itemize}
\item the labels of vertices on the line are defined accordingly to the orientation of the line, and
\item the orientations of the edges are defined so that the label of the initial vertex of an edge is smaller than
that of the terminal one.
\end{itemize}

\begin{figure}[htb]
\begin{center}
\includegraphics[width=12cm, height=4cm]{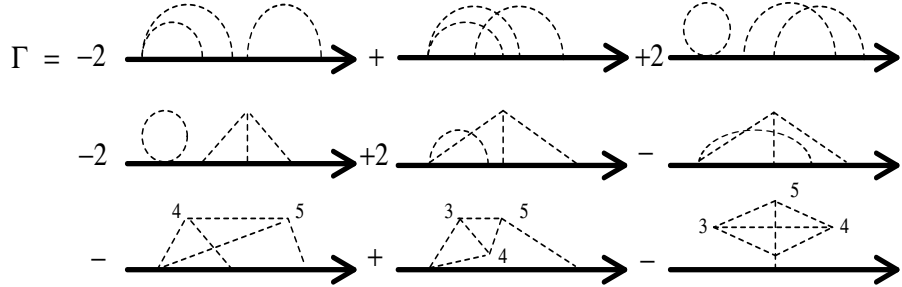}
\end{center}
\caption{a generator $\Gamma \in H^{3,1}(\D^*)$}\label{cocycle}
\end{figure}

\begin{rem}
It can be easily seen that $H^{3,l}(\D^* )=\{ 0\}$, $l \ge 4$.
The author has not computed $H^{3,l}(\D^* )$, $l=2,3$.
But Turchin's computation \cite{Tourtchine04_2} of certain spectral sequence related to $\D^*$ suggests that
$H^{3,l}(\D^* )$ might be zero for $l=2,3$.
In fact $H^{3,0}(\D^* )\cong \R$ and the Euler characteristic of the complex $\D^{3,*}$ is zero, so
$\rank H^{3,2}(\D^* )-\rank H^{3,3}(\D^* )=0$.
Thus there would be no contradiction even if $H^{3,2}(\D^* )=H^{3,3}(\D^* )=0$.
Of course it is not difficult to compute $H^{3,l}(\D^* )$, $l=2,3$, though it would be exhausting.\qed
\end{rem}

\section{Evaluation}

Suppose $n>3$ is odd and let $\Gamma \in H^{3,1}(\D^* )$ be the cocycle in Figure \ref{cocycle}.
Recall $\Lambda \in H_{3n-8}(\K )$ from \S \ref{action}.
The following theorem proves our main result.

\begin{thm}
The pairing $\ang{I(\Gamma )}{\Lambda}$ is not zero.
\end{thm}

\begin{proof}
We name the nine graphs in Figure \ref{cocycle} $\Gamma_1 ,\dots ,\Gamma_9$ respectively;
\[
 \Gamma =-2 \Gamma_1 +\Gamma_2 +2\Gamma_3 -2\Gamma_4 +2\Gamma_5 -\Gamma_6 -\Gamma_7 +\Gamma_8 -\Gamma_9 .
\]
First we remark the following fact \cite{CCL02};
the configuration space integral construction explained in \S \ref{BottTaubes} can be proceeded as long as
the volume form of $S^{n-1}$ is symmetric, that is, $i^* vol_{S^{n-1}} =-vol_{S^{n-1}}$ for the antipodal map
$i : S^{n-1} \to S^{n-1}$ (we are assuming $n$ is odd).
When $n>4$, the cohomology classes of $\K$ obtained via the configuration space integrals do not depend on the choice
of such symmetric volume forms.
So below we use the symmetric volume form whose support is localized in the (sufficiently small) neighborhood of
$(\pm 1 ,0,\dots ,0) \in S^{n-1} \subset \R^n$.

Let $\delta_i >0$ ($i=1,2$) be the `sizes' of resolutions of the self-intersections $z_i = f(\xi_i )=f(\xi_{i+2})$
($i=1,2$) of the immersion $f$ representing $v_2$ (see \S\S \ref{cycles}, \ref{action}).
We set $\delta_i =\varepsilon^2_i$, $i=1,2$ ($\varepsilon_i$ appears in the description of the resolution, see
\S \ref{cycles}).
We will compute the limit $\varepsilon_i \to 0$ of the pairing $\ang{I(\Gamma )}{\Lambda}$.
The homology class $[\Lambda ]$ is independent of the values $\varepsilon_i$, and so is the pairing.
But in the limit, as we will prove later, all the graphs $\Gamma_j$ except for $\Gamma_2$ do not contribute to
$\ang{I(\Gamma )}{\Lambda}$ (Lemmas \ref{1}, \ref{2}, \ref{3} and \ref{4}), and $\ang{I(\Gamma_2 )}{\Lambda}$
is not zero (Lemmas \ref{5}, \ref{6}).
\end{proof}

\begin{lem}\label{1}
In the limit, $\ang{I(\Gamma_j )}{\Lambda} \to 0$ for $j=7,8,9$.
\end{lem}

\begin{proof}
As in the computation in \S \ref{ex_evaluation}, only the integration over
$\mathrm{Conf}_0 [\R^1 ,v_i ] \setminus C_{\varepsilon}$ contributes to the above pairing in the limit.
So the graphs $\Gamma_j$ with less than four interval vertices never contribute to the pairing in the limit
$\varepsilon_i \to 0$.
\end{proof}

\begin{lem}\label{2}
In the limit, $\ang{I(\Gamma_j )}{\Lambda} \to 0$ for $j=4,5$.
\end{lem}

\begin{proof}
The graphs $\Gamma_4$ and $\Gamma_5$ have four vertices on the special line.
So the corresponding points $(t_1 ,\dots ,t_4 )$ is in $\mathrm{Conf}_0 [\R^1 ,4] \setminus C_{\varepsilon}$
if and only if $\abs{t_i -\xi_i} \le \varepsilon_i$, $1 \le i \le 4$.
Then in the case of $\Gamma_4$, the integrand $\theta_{11}$ is zero since we take the immersion $f$ so that
$D_1 f(t_1 )$ with $\abs{t_1 -\xi_1} \le \varepsilon_1$ cannot be near $(\pm 1,0,\dots ,0)$,
the support of our volume form.
In the case of $\Gamma_5$, the integrand $\theta_{12}$ also vanishes by similar reason.
\end{proof}

\begin{lem}\label{3}
In the limit, $\ang{I(\Gamma_6 )}{\Lambda} \to 0$.
\end{lem}

\begin{proof}
The points $(t_1 ,\dots ,t_4 ) \in \mathrm{Conf}_0 [\R^1 ,4] \setminus C_{\varepsilon}$ corresponding to interval
vertices should be as in the above Lemma.
So the integrand $\theta_{15}\theta_{25}\theta_{45}$ vanishes unless the point $x_5$ corresponding to the free vertex
$5$ is `near $(\pm \infty ,0,\dots ,0)$,' since otherwise the images of $\varphi_{i5} \circ \tilde{ev}$, $i=1,2,4$,
cannot be in the support of $vol_{S^{n-1}}$ simultaneously.

Now we look at two maps $\varphi_{i5} \circ \tilde{ev}$, $i=2,4$.
In the limit $\varepsilon_2 \to 0$, the points $f(t_2 ) \in S_2$ and $f(t_4 ) \in I_2$ are very near
($S_2 \approx S^{n-2}$ and the interval $I_2$ have been introduced in Remark \ref{linkingnumber}), and the free point
$x_5$ has to be far from them.
So the image of the map
\[
 (\varphi_{25} \circ \tilde{ev}) \times (\varphi_{45} \circ\tilde{ev}) :
 S_2 \times I_2 \longrightarrow (S^{n-1})^2
\]
is near the diagonal set $\Delta =\{ (v,v) \in (S^{n-1})^2\}$.
More precisely, for any open neighborhood $U$ of $\Delta$, there exists $\epsilon_0 >0$ such that the image of
$(\varphi_{25} \circ \tilde{ev}) \times (\varphi_{45} \circ\tilde{ev})$ is contained in $U$ for any
$\varepsilon_2 <\varepsilon_0$.

Thus on $\mathrm{Conf}_0 [\R^1 ,4] \setminus C_{\varepsilon}$ the integrand $\theta_{45}$ can be written as
\[
 \theta_{45}=(\varphi_{25} \circ \tilde{ev} +\varepsilon' \varphi )^* vol_{S^{n-1}}
\]
for the Gauss map $\varphi : S_2 \times I_2 \to S^{n-1}$ and some $\varepsilon' >0$ such that
$\varepsilon \xrightarrow{\varepsilon_2 \downarrow 0} 0$.
Hence the integrand is
\[
 \theta_{13}\theta_{15}\theta_{25}\theta_{45} =\varepsilon' \theta_{13}\theta_{15}\theta_{25}\varphi^* (vol_{S^{n-1}})
\]
and its integration converges to zero as $\varepsilon_2 \to 0$.
\end{proof}

\begin{lem}\label{4}
In the limit, $\ang{I(\Gamma_j )}{\Lambda} \to 0$ for $j=1,3$.
\end{lem}

\begin{proof}
First we prove $\ang{I(\Gamma_1 )}{\Lambda} \to 0$; the integrand $\theta_{12}$ is not zero only if
$t\in \mathrm{Conf}_0 [\R^1 ,4] \setminus C$ is such that $t_1$ is near $\xi_1$ and $t_2$ is near $\xi_3$.
But then no other $t_i$ can be near $\xi_2$.

$\ang{I(\Gamma_3 )}{\Lambda} \to 0$ since if $t\in \mathrm{Conf}_0 [\R^1 ,4] \setminus C$ then $t_1 \le \xi_1$ and
thus $\theta_{11}=D_1^* vol$ is always zero by our choice of $f$.
\end{proof}

\begin{lem}\label{5}
The limit of $\ang{I(\Gamma_2 )}{\Lambda}$ is not zero.
\end{lem}

\begin{proof}
The integrand $\theta_{13}\theta_{14}\theta_{25}$ does not vanish only if the
$t \in \mathrm{Conf}_0 (\R^1 ,5) \setminus C$ is such that $t_1$ is near $\xi_1$, $t_3 ,t_4$ are near $\xi_3$,
$t_2$ is near $\xi_2$ and $t_5$ is near $\xi_4$.
Integration with respect to $t_2 ,t_5$ and $u_2 \in S^{n-3}$ is the linking number, and hence equals to one
(Remark \ref{linkingnumber}).
So it remains to compute the integration with respect to $t_1 ,t_3 ,t_4$ and $u_1 \in S^{n-3}$.
We reformulate the situation around $z_1$ as follows (see Figure \ref{M});
\begin{itemize}
\item a point $P_1$ (corresponding to $f(t_1 )$) is on
\[
 M := \{ x^2_1 +\dots +x^2_{n-2}+x^2_n =1,\  x_{n-1} =0 \}
\]
(this sphere corresponds to $S_1 \approx S^{n-2}$ introduced in \S \ref{ex_evaluation}),
\item two points $(P_4 ,P_3 )\in \mathrm{Conf}_0 [\R^1 ,2]$ (corresponding to $f(t_4 ),f(t_3 )$) are on the
$x_{n-1}$-axis (corresponding to the interval $I_1$), and
\item the frame $e[s,u]\in SO(n-1)$ ($[s,u]\in \Sigma S^{n-3}$) acts on $\R^n$, fixing the $x_n$-axis.
\end{itemize}

\begin{figure}[htb]
\begin{center}
\includegraphics{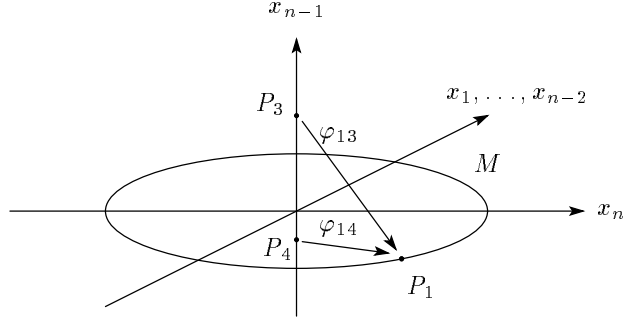}
\end{center}
\caption{integration around $z_1$}\label{M}
\end{figure}

If we define
\[
 F : N \longrightarrow (S^{n-1})^2 \quad (N:=\Sigma S^{n-3} \times S^{n-2} \times \mathrm{Conf}_0 [\R^1 ,2])
\]
by
\begin{align*}
 F([s,u],P_1 ,(P_4 ,P_3 )) &:= (\varphi_{13},\varphi_{14})( e[s,u]P_i )_{i=1,3,4} \\
 &= \left( e[s,u] \frac{P_1 -P_3}{\abs{P_1 -P_3}},\, e[s,u] \frac{P_1 -P_4}{\abs{P_1 -P_4}} \right),
\end{align*}
then our aim is to compute the integral
\[
 \int_{\Sigma S^{n-3} \times S^{n-2} \times \mathrm{Conf}_0 [\R^1 ,2]} F^* (vol_{(S^{n-1})^2})
\]
where $vol_{(S^{n-1})^2} = pr_1^* vol_{S^{n-1}} \wedge pr_2^* vol_{S^{n-1}}$ is a top form of $(S^{n-1})^2$
(remember the support of $vol_{S^{n-1}}$ is localized in the neighborhood of $(\pm 1,0,\dots ,0)$).

The map $F$ has its image in
\[
 A := \{ (x,y)\in (S^{n-1})^2 \, |\, x_n y_n \ge 0 \}
\]
where $x_n$ is the $n$-th coordinate of $x\in S^{n-1} \subset \R^n$.
In fact, as we will prove in Lemma \ref{6}, the map $F$ is two-fold covering on $\Int A$ branched on the diagonal
$\Delta = \{ (v,v)\in (S^{n-1})^2 \}$, thus the image of $F$ covers the half of the support of $vol_{(S^{n-1})^2}$
twice.
Moreover, for any $p,q \in F^{-1}(\Int A \setminus \Delta )$ with $F(p)=F(q)$, the map $G: T_p N \to T_q N$
defined by the commutative diagram
\[
 \xymatrix{
  T_p N \ar[dr]^-{F_*}_-{\cong} \ar[rr]^-{G}_-{\cong} & & T_q N \ar[ld]^-{\cong}_-{F_*} \\
   & T_{F(p)}(S^{n-1})^2 &
 }
\]
is an isomorphism of oriented tangent spaces for suitable orientations of $N$ and $(S^{n-1})^2$.
Thus the limit of the above integral is $\pm \frac{1}{2} \times 2 =\pm 1$.
\end{proof}

\begin{lem}\label{6}
$F|_{\Int N}$ is two-fold smooth covering onto $\Int A \setminus \Delta$ and $G$ is an orientation preserving map.
\end{lem}

\begin{proof}
We denote by $v_n$ the $n$-th coordinate of $v\in \R^n$.
We will show that for any $(v_3 ,v_4 ) \in \Int A \setminus \Delta$ (then $(v_3 )_n (v_4 )_n >0$),
we can find $([s,u],P_i )$ so that $F ([s,u],P_i )=(v_3 ,v_4 )$, that is,
\begin{equation}\label{v}
 e[s,u]\frac{P_1 -P_i}{\abs{P_1 -P_i}}=v_i ,\quad i=3,4.
\end{equation}
Consider the 2-plane $H(v_3 ,v_4 )\subset \R^n$ spanned by two vectors $v_3 ,v_4$.
Then the intersection
\[
 l(v_3 ,v_4 ) := H(v_3 ,v_4 ) \cap \{ x_n =0\}
\]
is a one-dimensional linear subspace of $\{ x_n =0\}$.

Since $([s,u],P_i )$ should satisfy (\ref{v}), the points $e[s,u]P_i$ ($i=3,4$) should be on $l(v_3 ,v_4 )$
and $e[s,u]P_1$ should be on $H(v_3 ,v_4 )\cap e[s,u]M$.
So the frame $(s,u)$ should transpose $x_{n-1}$-axis to $l(v_3 ,v_4 )$.
There are two such frames, namely, $[s,u]$ and $[-s,-u]$ for some $[s,u]\in \Sigma S^{n-3}$.
We have $e[s,u]M=e[-s,-u]M$, and this sphere intersects with $H(v_3 ,v_4 )$ at two points.
One has positive $n$-th coordinate and the other has negative one.
When $(v_3)_n >0$ (resp.\ $(v_3)_n <0$), we choose positive (resp.\ negative) one and name it $e[s,u]P_1$.
Then $P_3$ and $P_4$ are determined uniquely so that $\varphi_{1i} (P_1 ,P_3 ,P_4 )=e[s,u]^{-1}v_i$.

Thus we have two points $(\pm [s,u],P_i )$ which are mapped to $(v_3 ,v_4 )\in (S^{n-1})^2$ via $F$.
The map $F$ is clearly smooth.
The above arguments show that $F^{-1}$ is also smooth, hence $F|_{\Int N}$ is locally diffeomorphic
two-fold covering.

The map $G$ is orientation preserving, since it is essentially the antipodal map
$\Sigma S^{n-3} \to \Sigma S^{n-3}$ and it preserves orientation (we assume $n$ is odd).
\end{proof}

\begin{rem}
In general, nothing is known about $H^{k,1}(\D^* )$, $k \ge 4$.
But anyway suppose we have $\Gamma =\sum a_i \Gamma_i \in H^{k,1}(\D^* )$.
Let $\Gamma'$ be a chord diagram with $(k-1)$ chords.
Then, in a similar way as above, we can compute the pairing
$\ang{I(\Gamma )}{r_* \lambda (\alpha (\Gamma' ),e)}$;
choose an immersion $f$ representing $\alpha (\Gamma' )$ so that almost all of the image of $f$ lies in
$x_{n-1}x_n$-axis.
We proceed the configuration space integral construction by using the symmetric volume form of $S^{n-1}$
whose support is localized in the neighborhood of $(\pm 1,0,\dots ,0)$.
Let $\delta_i >0$ be the `size' of the $i$-th resolution of the immersion $f$.
Then, in the limit $\delta_i \to 0$ ($1\le i\le k-1$), only the graphs $\Gamma_i$ obtained from $\Gamma'$
with one of its chord `doubled' contribute to the pairing (see Figure \ref{doubled}).

\begin{figure}[htb]
\begin{center}
\includegraphics[width=9cm, height=1cm]{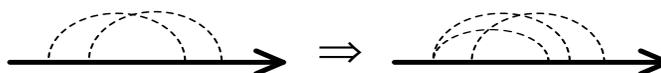}
\end{center}
\caption{doubling operation}\label{doubled}
\qed
\end{figure}
\end{rem}

\providecommand{\bysame}{\leavevmode\hbox to3em{\hrulefill}\thinspace}
\providecommand{\MR}{\relax\ifhmode\unskip\space\fi MR }
\providecommand{\MRhref}[2]{%
  \href{http://www.ams.org/mathscinet-getitem?mr=#1}{#2}
}
\providecommand{\href}[2]{#2}

\end{document}